\documentclass[pdflatex]{sn-jnl}

\usepackage{graphicx}%
\usepackage{multirow}%
\usepackage{amsmath,amssymb,amsfonts}%
\usepackage{amsthm}%
\usepackage{mathrsfs}%
\usepackage[title]{appendix}%
\usepackage{xcolor}%
\usepackage{textcomp}%
\usepackage{manyfoot}%
\usepackage{booktabs}%
\usepackage{algorithm}%
\usepackage{algorithmicx}%
\usepackage{algpseudocode}%
\usepackage{listings}%
\usepackage{natbib}
\usepackage{hyperref}
\usepackage{multibib}
\usepackage{fullpage}
\usepackage{lineno}
\newtheorem{lemma}{Lemma}
\newtheorem{proposition}{Proposition}

\newcites{App}{References}

\theoremstyle{thmstyleone}%
\newtheorem{theorem}{Theorem}

\theoremstyle{thmstyletwo}%

\theoremstyle{thmstylethree}%

\begin{document}

\title{Regional factors determine the feasibility of an elimination strategy}


\author[1]{\fnm{George} \sur{Adu-Boahen}}
\author[2]{\fnm{Troy} \sur{Day}}
\author[3]{\fnm{Michael J.} \sur{Plank}}
\author[1,4*]{\fnm{Amy} \sur{Hurford}}


\affil[1]{Mathematics \& Statistics Department, Memorial University, Canada}
\affil[2]{Mathematics \& Statistics Department, Queen's University, Canada}
\affil[3]{School of Mathematics \& Statistics, University of Canterbury, New Zealand}
\affil[4]{Biology Department, Memorial University, Canada}
\affil[*]{Corresponding author: ahurford@mun.ca}



\abstract{Assessment of whether an elimination or suppression strategy might be implemented during a pandemic have been made for specific countries, but combinations of control measures, and how regional factors may affect public health strategy decisions has been less explored. We provide an explicit mathematical characterization of regional characteristics and the resulting epidemiology to evaluate whether elimination or suppression is feasible using targeted measures, such as case isolation and contact tracing. We derive an epidemic model where community members can be infected by either travellers or other community members, and with constraints on the public health resources available to support isolation of arriving travellers and infected community members. We prove that the optimal controls are to immediately implement public health measures at their maximum levels. We find that elimination or suppression strategies are feasible with targeted measures if public health capacity is high, the transmission rate is low, or if the arrival rate of infected travellers is low. Our results illustrate that a particular country that implemented a mitigation strategy during the pandemic may not necessarily have achieved better outcomes if it had instead implemented elimination because elimination may not have been feasible in countries with low public health capacity to support control measures, high transmission rates, or high traveller arrival rates. Combinations of regional characteristics determine whether an elimination, suppression or mitigation strategy without resurgence is feasible. We find that regions with the low traveller arrival rates and the capacity to implement and enforce isolation or quarantine of arriving travellers are less likely to exceed their capacity to contact trace and isolate community members. Travel measures prevent only a small number of community infections via their direct effect, but indirectly, by protecting contact tracing capacity, travel measures can prevent epidemic resurgence and substantially reduce the number of people infected during a pandemic before a vaccine or therapy is developed.
}


\maketitle

\section{Introduction}\label{sec:intro}

The public health response of different countries during the COVID-19 pandemic has been compared and discussed in terms of whether an elimination or mitigation strategy was implemented and how many severe infections were averted \citep{peng2023relative}. An elimination strategy, also referred to during the pandemic as a `zero-COVID' strategy, aims to reduce infections to zero \citep{baker_elimination_2020, metcalf_challenges_2021, heywood_eradication_2020, wu_aggressive_2021}. Whether an elimination strategy is best, or even possible, depends on characteristics of the country \citep{martignoni_is_2024,french2025preparing}. We investigate the role of regional characteristics to better understand how the feasibility of an elimination strategy can differ between regions and to inform the interpretation of between-country comparisons of pandemic response strategies. 

During a pandemic, an elimination strategy often, but not always, requires strong border controls to minimize the number of imported infections as failure to contain these imported infections can lead to community outbreaks. A proposed definition of elimination during a pandemic is that incidence of community-acquired infections is zero, and that jurisdictions would maintain their elimination status if infections occur only in arrivals that are isolating \citep{baker_elimination_2020}.  Suppression aims to reduce disease incidence to low levels but allows for some transmission. Mitigation strategies aim to slow the spread of an infectious disease to reduce the health impact and avoid overwhelming healthcare capacities \citep{baker2023covid, wu_aggressive_2021,thompson2026infectious}. Typically, mitigation involves flattening the epidemic curve and eventually reaching the herd immunity threshold, where a sufficient proportion of the population has been infected, such that control measures can be relaxed without epidemic resurgence. In contrast, unless the disease can be eradicated globally or vaccines can provide sterilizing immunity, elimination and suppression are typically temporary strategies that minimize the number of infections until an alternative strategy becomes acceptable \citep{french2025preparing}. This could be the development of a vaccine or therapy that is effective at reducing transmission and/or disease severity to an acceptable level. These definitions of possible pandemic response strategies are not mutually exclusive, but characterizing a combination of public health measures as elimination, suppression, or mitigation is helpful to define the strategic aims and set the public's expectations for what the results of complying with public health measures are expected to be, which supports transparent communication and public trust.

Regional factors impact the best COVID-19 public health strategies because they impact the epidemiology and whether an elimination strategy is feasible in terms of the economic cost, and healthcare infrastructure and societal implications \citep{martignoni_is_2024}. Implicit in the interpretation of some between-country comparisons of the COVID-19 pandemic response is that countries could have reduced COVID-19 mortality if they had adopted the same public health strategy as countries that achieved low mortality due to COVID-19. Such comparisons are based on definitions of `peer countries', for example those with similar social and political systems, per capita income, and population size \citep{razak}. This definition overlooks the amount of cross-border travel and trade, which affects the cost of travel measures and the arrival rate of imported infections when no travel measures are implemented, and both of these factors will affect which strategies are optimal and feasible. The approach of defining peer countries does not make explicit how the characteristics of peer countries impacts the epidemiology of these regions or how these characteristics interact.

We will show that the feasibility of different pandemic response strategies depends on characteristics of the country or region, and that the assessment of how many infections could have been averted if other pandemic response strategies were implemented is best conducted by epidemic scenario modelling that accounts for these characteristics. Describing the spread of infections with a mathematical model enables quantitative handling of key considerations that may differ between regions, such as the importation rate or local transmission rate. The scenario modelling approach is preferable to simply comparing peer countries, since the peer countries may be similar but not be identical, and the mathematical model describes the epidemiological impact of region-specific factors and their interrelatedness.

The goal of our study is to investigate epidemiological and public health factors that influence the feasibility of eliminating or suppressing a pandemic infectious disease with targeted control measures, such as case isolation and contact tracing. We consider a model for control of an epidemic in the presence of imported, travel-related, infections and with constraints on the public health resources available to support testing and isolation of cases in the community and of recently arrived travellers. Our results show that region-specific factors play an important role, implying that while elimination or suppression may be desirable strategies in some jurisdictions, they may not be feasible in others.  We wish to highlight that generalities such as that an elimination strategy is the best approach, or that there are no circumstances in which travel measures are appropriate to control a pandemic, are not helpful for some jurisdictions and there needs to be better contextualization of recommendations.

\section{Model}\label{sec:model}
\subsection{Epidemiological and control resource consumption dynamics}\label{sec:epi}
We consider a deterministic compartment-based model for epidemiological and control resource consumption dynamics on a time interval $t = [0,T]$. These dynamics are, 
\begin{align}
    \dot{S} &=-\beta S(I_1 + I_2) \label{mm1}, \\
    \dot{I}_1 &=\beta S(I_1 + I_2) -(\mu + u_1) I_1 \label{mm2},\\
    \dot{I_2} &= p\theta (1-u_2) -\gamma I_2 \label{mm},\\
    \dot{U}_{1} &= u_1I_1, \\
    \dot{U}_{2} &= \theta  u_2, \label{eq:U2-}
\end{align}
 where $S(t)$ is the number of susceptible community members, $I_1(t)$ is the number of infectious community members, and $I_2(t)$ is the number of infectious travellers that do not self-isolate. The force of infection term assumes that susceptible community members are infected at a rate $\beta>0$ per infected community member or per infected traveller that does not self-isolate per unit time, and the rate that community members become uninfectious is $\mu>0$ per unit time. The rate that infected travellers arrive in the community is $p\theta$, where $\theta > 0$ is the arrival rate of travellers per unit time and $p \in (0,1)$ is the proportion of arriving travellers that are infected. The rate that infectious travellers become uninfectious is $\gamma$, with $\gamma > \mu$ because community members were in the community from their first day of infectiousness, while travellers may have spent some of their infectious period outside the community. The initial conditions for the epidemiological variables are all positive, i.e., $S(0)>0$, $I_1(0)>0$, and $I_2(0)>0$.

There are several different ways that importations might be included in an epidemic model. Our formulation assumes that travellers are a source of infection for susceptible members of the community, but does not explicitly consider the origin of the infected travellers or any infectious disease dynamics at the origin. Further, our formulation assumes that the community members themselves do not become travellers. This model formulation was chosen because we can examine the risk of infection from a non-community source without adding substantial model complexity. Further, we do not consider births and background mortality of community members. 

The control variables in equations \eqref{mm1}-\eqref{eq:U2-} are $u_1(t)$ and $u_2(t)$, which, respectively, represent the daily isolation rate per infected community member and the probability that a traveller isolates, which is assumed to be the same for infectious and uninfectious travellers. This models a quarantine policy that applies equally to all travellers regardless of their infection status. The community control, $u_1(t)$, is modelled as a rate because it takes some time for infectious community members to be contact traced and isolated. The traveller control, $u_2(t)$, is modelled as a probability because travellers, on arrival, know that they are expected to isolate or quarantine immediately. We assume that initially none of these community test-trace-isolate, $U_1(0)=0$,  and traveller isolation, $U_2(0)=0$, resources have been used. Our formulation assumes that both infected and uninfected travellers are required to self-isolate, and so $\theta$ (the traveller arrival rate), rather than $p\theta$ (the \emph{infected} traveller arrival rate), appears in equation~\eqref{eq:U2-}. We assume (without proof) that the state variables in equations \eqref{mm1}-\eqref{eq:U2-} are non-negative for all time.

 An elimination strategy can involve the implementation of travel measures in combination with community measures, and so we need to consider both of these controls in our model. Note that our model considers targeted control measures, such as community testing, case isolation and contact tracing (referred to as test-trace-isolate or TTI) and traveller quarantine or self-isolation, but does not consider the application of broad public health and social measures (PHSMs), such as school and business closures, gathering restrictions, and stay-at-home orders that would affect the community transmission rate.

\subsection{Control resource constraints}\label{sec:constraints}
The controls are required to be piecewise continuous functions,
\[ u_1: [0,T] \rightarrow [0, u_{1\text{max}}] \qquad \qquad u_2: [0,T] \rightarrow [0, u_{2\text{max}}], \]

where $0$ corresponds to no control and $u_{1\text{max}}>0$ and $u_{2\text{max}} \in (0,1)$ correspond to the maximum values of $u_1(t)$ and $u_2(t)$ respectively.

We assume that the total resources available to control the outbreak are constrained such that,
\begin{eqnarray}
   U_{1}(T) \leq U_{1\text{max}} & \qquad \mbox{and} \qquad & U_{2}(T) \leq U_{2\text{max}} \label{rc2}
\end{eqnarray}
with $U_{i\text{max}}>0$. Examples of such resources for $U_{1}(T)$ are funding to pay the staff employed in testing, tracing, and isolating infected community members, financial support to those required to isolate, resources such as test consumables, and social license for the contact tracing and isolation measures themselves. Examples of resources for $U_{2}(T)$ include funding to pay the staff involved in implementing and enforcing post-arrival travel measures, as well as the necessary resources, such as isolation facilities or testing equipment to complete these activities, and the social license that affords the traveller's willingness to comply.

\subsection{Cost function}\label{sec:cost}
We assume that the aim of public health measures is to minimize the cumulative number of infections in the community outbreak over some period of time $T$,
\begin{align}
    J = \int_0^{T} \beta S(I_{1} +  I_{2}) \,\, dt, \label{obj}
\end{align}
subject to the epidemiological and control resource consumption dynamics (equations \eqref{mm1}-\eqref{eq:U2-}) and the control resource constraints (equations \eqref{rc2}).

\subsection{Explanation of how regional characteristics are represented}
Given the parameters that we have defined, we consider local public health capacity as being represented by the maximum daily community member isolation rate, $u_{1\text{max}}$, the maximum fraction of travellers that comply with self-isolation requirements, $ u_{2\text{max}}$, and the total amount of resources that are available to support isolation of community members, $U_{1\text{max}}$, and travellers, $U_{2\text{max}}$. We view the community transmission rate, $\beta$, to reflect characteristics of the local community, such as contact rates and the frequency of large gatherings in poorly ventilated spaces. We view the traveller arrival rate, $\theta$, as another regional characteristic, for example, where remote, inaccessible and undesirable regions would have fewer travellers arriving per day. The aim of our analysis is to formulate a model that can explore how the feasibility of public health strategies can depend on regional public health capacity and regional characteristics.

\subsection{Explanation of the fixed $T$ assumption}
We run the model for a fixed time horizon $T$, representing the time at which an effective vaccine or therapy becomes available, effectively ending the epidemic immediately. This contrasts with a previous approach that has defined the end of the epidemic to be when the number of infectious people falls below some threshold \citep{hansen_optimal_2011}. However, the latter is not appropriate for our model because even if $I_1(t)$ is very small, the outbreak cannot be considered `over', because the ongoing imported infections mean that there is always a future risk of epidemic resurgence if the susceptible population is sufficiently large. 

Assuming a fixed time horizon is clearly a simplification: in reality, any vaccines or treatment will be imperfect and will not be accessed by 100\% of the population concurrently. Therefore, there will still be costs incurred after the vaccine or treatment becomes available and potentially new trade-offs to consider between the costs of implementing controls and the costs of infection. Nonetheless, working with a fixed time horizon $T$ is a reasonable simplification that allows for the possibility of strategies that eliminate or suppress the epidemic until a vaccine or treatment is available. One drawback of this modelling approach is that it can reward strategies that push infections that would otherwise occur just before time $T$ beyond the time horizon. This is not necessarily a realistic or desirable strategy, and we interpret our results with this in mind. For the most part, the strategies that we investigate tend to result either in an epidemic wave that subsides before time $T$, or in elimination or suppression of the epidemic until time $T$. We believe that these are reasonable alternatives to consider, and describe these in more detail in Section \ref{sec:PH_strategies}.

\subsection{Interpreting the deterministic model formulation}

In situations where the number of infectious individuals is small (e.g. when an outbreak is close to elimination), a branching process model or stochastic individual-based model is normally preferable to a deterministic model, as they accurately capture that if the number of infected individuals falls to zero, the epidemic cannot rebound in the absence of infection re-introduction from outside the population. The model we have used here has the limitation that the infectious compartments $I_1(t)$ and $I_2(t)$ can be arbitrarily small, but are always strictly positive. This means that the epidemic is never completely eliminated and can rebound from unrealistically low values of $I_1(t)$ and $I_2(t)$, which could represent fewer than one individual in a finite population. 

As a consequence, this model would not be suitable for investigating the elimination of outbreaks in the absence of imported cases. However, we argue that the model is a reasonable tool for the problem investigated here because it captures the epidemic rebound that is expected to occur as a result of imported infections when controls are relaxed and the susceptible population is sufficiently large. Thus, when $0<I_1(t)<1$ or $0<I_2(t)<1$ in our model, we interpret this as an average prevalence, representing periods of zero prevalence being punctuated by sporadic community outbreaks that are initiated by infected travellers and are eventually brought under control. Likewise, the variables $U_1(t)$ and $U_2(t)$ track the cumulative resources expended in controlling community outbreaks initiated by travellers in an average sense as opposed to increasing in discrete jumps. Provided $U_1(t)<U_{1\mathrm{max}}$ (and community controls are strong enough to bring the reproduction number below $1$), the cost of maintaining elimination status is relatively low (but non-zero) because the number of community infections that need to be isolated is small.
Future work using a stochastic model to address these issues would be helpful, although this would limit the potential for analytical insights and require a more computationally intensive approach. The deterministic model we present here is useful as it describes the structure that a stochastic model could be based upon and could inform the interpretation of any future stochastic modelling.

\subsection{Definitions of public health strategies} \label{sec:PH_strategies}

Our definitions of different public health strategies are provided in Table \ref{tab:strategy-defns}. We define the distinction between elimination/suppression and mitigation as depending on the control reproduction number, $R_0^C$, which for our system of equations \eqref{mm1}-\eqref{mm} is defined as
\begin{equation}
R_0^C = \frac{\beta S(0)}{\mu + u_{1\text{max}}},
\end{equation}
and is the number of secondary infections generated by an infectious individual in a wholly susceptible population when community controls are implemented at their maximum rate. The control reproduction number, $R_0^C$, is a threshold condition for local stability of the disease-free equilibrium assuming no community infections due to travellers.

In our deterministic model formulation, it is difficult to clearly distinguish between elimination and suppression strategies due to the continuous presence of imported infections. We therefore categorize outcomes primarily as being either mitigation ($R_0^C>1$) or elimination/suppression ($R_0^C<1$), with a secondary categorization of whether or not resurgence occurs (determined by whether or not resources are exhausted before time $T$). 

\begin{table}[h]
        \centering 
        \caption{Definition of public health strategies for community isolation}
        \begin{tabular}{p{2cm}p{12cm}}
        \toprule
            \textbf{Public health strategy} & \textbf{Description} \\
            \midrule
           Elimination & Aims to reduce community infection prevalence to zero locally, but not in all regions, such that there remains a risk of disease importation \citep{baker_elimination_2020, metcalf_challenges_2021}. Defined as $R_0^C < 1$. Our definition of elimination also requires that community isolation measures are sufficient to remain in place for $T$ days, which is the condition that $U_{1}(t) \leq U_{1\text{max}}$ $\forall t \in [0,T]$.\\
            \midrule
            Suppression & Aims to reverse epidemic growth \citep{ferguson_report_2020} and bring the number of cases down to a low level \citep{heywood_eradication_2020, james_suppression_2020}, noting that community transmission may still take place \citep{wu_aggressive_2021, baker_elimination_2020, heywood_eradication_2020}. Defined as $R_0^C < 1$ and $U_{1}(t) \leq U_{1\text{max}}  \forall t \in [0,T]$. Note that given our deterministic model formulation, we cannot distinguish between elimination and suppression.\\
            \midrule
            Mitigation & Aims to slow the spread of an infectious disease to avoid overwhelming healthcare capacities and to reduce overall morbidity and mortality \citep{baker2023covid, wu_aggressive_2021}. Defined as control measures do not reverse epidemic growth, $R_0^C \geq 1$. Our definition of mitigation also requires that community isolation measures are sufficient to remain in place for $T$ days, which is the condition that $U_{1}(t) \leq U_{1\text{max}}  \forall t \in [0,T]$.  \\
            \midrule
            Resurgence after elimination or suppression & Public health measures, when implemented, result in elimination or suppression ($R_0^C < 1$), but community isolation resources are insufficient to be used for the entire outbreak and infections increase when the community isolation resources are exhausted. Defined as $U_{1}(t) = U_{1\text{max}}$ for some $t<T$. Removal of control when the population is mostly susceptible will always result in epidemic resurgence for our model. \\
            \midrule
            Resurgence after mitigation & Public health measures, when implemented, result in mitigation ($R_0^C \geq 1$), but community isolation resources are insufficient to be used for the entire outbreak and infections increase when the community isolation resources are exhausted. Defined as $U_{1}(t) = U_{1\text{max}}$ for some $t<T$. \\
            \botrule
        \end{tabular}
        \label{tab:strategy-defns}
    \end{table}

\subsection{Numerical methods}

To verify the analytical result above, we numerically determine the optimal controls by solving the boundary value problem involving the transversality conditions and using the \texttt{CTDirect} method from the \texttt{OptimalControl} package for Julia \citep{Caillau_OptimalControl_jl_a_Julia}. We characterized the epidemiological dynamics for specific controls, $u_1(t)$ and $u_2(t)$, using the \texttt{DeSolve} package for R \citep{Soetaert2010deSolve}. Julia and R code to run the model and reproduce the results in this article are publicly available at \href{here}{https://github.com/ahurford/elimination-mitigation}.

\begin{table}[h]
        \centering 
        \caption{Default values of parameters and variables. Simulation results use these default values unless otherwise stated.}
        \begin{tabular}{p{2cm}p{5cm}p{5cm}}
        \toprule
            \textbf{Symbol} & \textbf{Description} & \textbf{Value} \\
        \botrule
        $R_0 = \beta S(0)/\mu$ & number of secondary infections in the wholly susceptible community when there are no control measures & 2 \\
            $\beta$ & transmission rate in the community & $R_0\mu/S(0)$ per individual per day \\
            $\mu$ & rate that community members become uninfectious & 1/10 per day \\
            $p$ & proportion of arriving travellers that are infected & 1/1000 \\
            $\theta$ & arrival rate of travellers & 500 individuals per day \\
            $\gamma$ & rate that travellers become uninfectious & $2\mu$ per day \\
            \midrule
            $u_{1\text{max}}$ & maximum daily isolation rate per infected community member & 0.15 per day \\
            $u_{2\text{max}}$ & maximum probability that a traveller isolates & 0.99\\
            $U_{1\text{max}}$ & number of community members that can be isolated before $T$ given total public health resources & 1000 individuals \\
            $U_{2\text{max}}$ & number of travellers that comply with isolation requirements before $T$ given total public health resources & $0.8 \times T \times \theta$ individuals \\
            \midrule
            $N$ & population size & 10,000 individuals\\
            $I_1(0)$ & initial infection prevalence in the community & 1 individual\\
            $S(0)$ & initial number of susceptible community members & $N - I_1(0)$ individuals\\
            $I_2(0)$ & initial infection prevalence in travellers & $p\theta/\gamma$ individuals \\
            $T$ & time until an effective vaccine or therapy becomes available & 365 days \\
            \botrule
        \end{tabular}
        \label{tab:params}
    \end{table}

\section{Results}

\subsection{Pontryagin's maximization principle to determine the optimal control}

We apply the theory of Pontryagin's maximization principle to determine the optimal control that minimizes the cumulative number of infections. To do this, we reframe the problem outlined in Sections \ref{sec:epi}-\ref{sec:cost} as the maximization of $-J$. The Hamiltonian for our problem is,
\begin{equation}
H= -\beta S (I_1 + I_2)\Psi -(\mu + u_1) I_1 \lambda_{I_1} + (p\theta (1-u_2) -\gamma I_2)\lambda_{I_2} +  u_1I_1\lambda_{U_1} + \theta u_2\lambda_{U_2}, 
\end{equation}
where $\Psi := \lambda_0 + \lambda_S - \lambda_{I_1}$, with the adjoint variables defined by the equations
\begin{eqnarray}
\dot{\lambda}_S & = & \beta (I_1+I_2)\Psi, \label{eq:ad1} \\
\dot{\lambda}_{I_1} & = &\beta S \Psi +  \mu \lambda_{I_1}  +u_1(\lambda_{I_1}-\lambda_{U_1}), \\
\dot{\lambda}_{I_2} & = & \beta S \Psi + \gamma \lambda_{I_2}, \\
\dot{\lambda}_{U_1} & = & 0, \\
\dot{\lambda}_{U_2} & = & 0, \label{eq:adend1} \\
\dot \lambda_0 &=&0. \label{eq:adend} 
\end{eqnarray}
Here $\lambda_0 \geq 0$ is the objective functional multiplier, the transversality conditions are $\lambda_S(T)$ = $\lambda_{I_1}(T)$ = $\lambda_{I_2}(T)$ = 0, and 
the Karash-Kuhn Tucker conditions (\citealt{kkt}, Theorem 5.8) give,
\[ \lambda_{U_1}(T) \leq 0, \quad \quad  \lambda_{U_1}(T)(U_1(T) - U_{1\text{max}})=0, \quad \quad \lambda_{U_2}(T) \leq 0, \quad \quad  \lambda_{U_2}(T)(U_2(T) - U_{2\text{max}})=0. \]
By constancy and the Karush-Kuhn Tucker Theorem, $\lambda_{U_1} \equiv \lambda_{U^*_1}\leq 0$ and $\lambda_{U_2} \equiv \lambda_{U^*_2} \leq 0$ where we have used asterisks to denote the value of $\lambda_{U_i}$ associated with the optimal control. 
The Hamiltonian is affine in $(u_1, u_2)$ and so we define the reduced switching functions
\[
\varphi_1(t)
:=
\lambda_{U_1}^*-\lambda_{I_1},
\qquad
\varphi_2(t)
:=
\lambda_{U_2}^*-p\lambda_{I_2}.
\]
These have the same signs as the corresponding coefficients of \(u_1\)
and \(u_2\) in the Hamiltonian since \(I_1>0\) and \(\theta>0\). The maximum condition then implies $u_i^*(t) \in \{0, u_{i\max}\}$ for almost every $t$ for which $\varphi_i(t) \ne 0$.
\begin{theorem}
The optimal control is given by
\[
u_i^*(t)
=
\begin{cases}
u_{i\max}, & 0\leq t<\tau_i,\\
0, & \tau_i<t\leq T,
\end{cases}
\qquad i=1,2.
\]
for some \(\tau_1,\tau_2\in[0,T]\). 
\end{theorem}
The proof of Theorem 1 makes use of the following two Lemmas:
\begin{lemma}\label{lem:nondegen}
The optimal control is non-degenerate; i.e., $\lambda_0 \neq 0$ and so $\lambda_0>0$.
\end{lemma}
\begin{proof}: First note that the equations for the non-constant adjoint variables $\lambda:=(\lambda_S, \lambda_{I_1}, \lambda_{I_2})$ can be written in matrix form as $\dot \lambda=A \lambda+\lambda_0 b+ \lambda_{U_1}c$ where $A$ (a $3 \times 3$ matrix) and $b$ (a vector) depend on the state variables, and $c=(0,-u_1(t),0)$. Now suppose, for contradiction, that $\lambda_0=0$. There are two cases with respect to the value of $\lambda_{U_1}$:
\begin{itemize}[leftmargin=*]
\item Case 1: $\lambda_{U_1}^*<0$. In this case $\varphi_1(T)=\lambda_{U_1}^* - \lambda_{I_1}(T)=\lambda_{U_1}^*<0$ and so $u_1(t)=0$ in a some final interval containing $T$ (because $\varphi_1(t)$ is continuous). During this interval the costate variables obey $\dot \lambda=A \lambda$ with terminal condition $\lambda(T)=(0,0,0)$ and so their solution is $\lambda(t)=(0,0,0)$ during this interval. This argument can be bootstrapped backwards in time to $t=0$, showing that $\lambda(t)=(0,0,0)$ for all $t \in [0,T]$, and thus $u_1(t)=0$ a.e. on $ [0,T]$. As a result, $U_1(T)=0$ and so $\lambda_{U_1}^*(0 - U_{1\max}) \neq 0$ which contradicts the condition that $\lambda_{U_i}(T)(U_i(T) - U_{i\max}) = 0$.
\item Case 2: $\lambda_{U_1}^*=0$. In this case the adjoint variables obey $\dot \lambda=A \lambda$ with terminal condition $\lambda(T)=(0,0,0)$ and so their solution is $\lambda(t)=(0,0,0)$ for all $t \in [0,T]$. Since all adjoint variables cannot be identically zero along the optimal control, we must have $\lambda_{U_2}^*<0$. As a result, $\varphi_2(t) = \lambda_{U_2}^* - p\, \lambda_{I_2}(t)= \lambda_{U_2}^*<0$ and so $u_2(t)=0$ a.e. on $[0,T]$. As a result, $U_2(T)=0$ and so $\lambda_{U_2}^*(0 - U_{2\max}) \neq 0$ which contradicts the condition that $\lambda_{U_i}(T)(U_i(T) - U_{i\max}) = 0$.
\end{itemize}
\end{proof}
\begin{lemma}\label{lem:identity}
On any subinterval on which $u_1$ and $u_2$ are continuous (and hence the state and adjoint variables are $C^1$) we have
\[
\frac{d}{dt}\bigl(\beta\, S\, \Psi\bigr) = -\beta\, S\, \dot\lambda_{I_1}.
\]
\end{lemma}
\begin{proof}:
Direct computation using $\dot S = -\beta S(I_1+I_2)$ and $\dot\Psi = \dot\lambda_0+\dot\lambda_S - \dot\lambda_{I_1} = \beta(I_1+I_2)\Psi - \dot\lambda_{I_1}$ gives
\[
\tfrac{d}{dt}(\beta S\Psi) = -\beta^2 S(I_1+I_2)\Psi + \beta S\bigl[\beta(I_1+I_2)\Psi - \dot\lambda_{I_1}\bigr] = -\beta S\,\dot\lambda_{I_1}. \qedhere
\]
\end{proof}
\subsubsection{Proof of Theorem 1}

Without loss of generality we normalize the adjoint variables such that $\lambda_0=1$ and so $\Psi=1+\lambda_S-\lambda_{I_1}$ and $\Psi(T)=1$. The Proof of Theorem 1 follows directly from Propositions 1 and 2 below.
\begin{proposition}[The control $u_1$]
\label{prop:u1-one-switch}
The switching function \(\varphi_1\) is strictly decreasing. Consequently,
\(u_1^*\) has at most one switch, necessarily from \(u_{1\max}\) to \(0\).
\end{proposition}
\begin{proof}:
The maximum condition gives
\[
u_1^*(t)
=
\begin{cases}
u_{1\max}, & \varphi_1(t)>0,\\
0, & \varphi_1(t)<0,
\end{cases}
\qquad\text{for a.e. }t.
\]
At points where \(\varphi_1=0\), the value of \(u_1^*\) is not determined by \(\varphi_1\) but the product \(\varphi_1u_1^*\) is always given by \begin{equation}
\label{eq:phi1-u1-product}
\varphi_1u_1^*
=
u_{1\max}\max\{\varphi_1,0\}
\qquad\text{a.e.},
\end{equation}
Differentiating $\varphi_1$ gives
\begin{align}
\dot\varphi_1
&=
-\dot\lambda_{I_1}\notag\\
&=
-\beta\, S\, \Psi-\mu(\lambda_{U_1}^*-\varphi_1)+\varphi_1u_1^*\notag\\
&=
-\beta\, S\, \Psi-\mu \lambda_{U_1}^*+\mu\varphi_1
+u_{1\max}\max\{\varphi_1,0\}.
\label{eq:phi1-first-order}
\end{align}
The right-hand side of \eqref{eq:phi1-first-order} is continuous and therefore \(\varphi_1\in C^1([0,T])\). Moreover, we can differentiate \eqref{eq:phi1-first-order} almost everywhere to obtain
\begin{align}
\ddot \varphi_1
&=
-\frac{d}{dt}[\beta\, S\, \Psi]+(\mu+u_{1\max}\mathbf{1}_{\{\varphi_1>0\}})\dot \varphi_1\notag\\
&=
\left[
\mu
+u_{1\max}\mathbf 1_{\{\varphi_1>0\}}
-\beta S
\right]\dot \varphi_1
\qquad\text{a.e.}
\label{eq:w1-linear}
\end{align}
Define
\[
A(t)
:=
\mu
+u_{1\max}\mathbf 1_{\{\varphi_1(t)>0\}}
-\beta S(t).
\]
Since \(A\in L^\infty(0,T)\), the scalar linear equation
\eqref{eq:w1-linear} can be written
\begin{equation}
\label{eq:w1-representation}
\dot \varphi_1(t)
=
\dot \varphi_1(T)
\exp\left(
-\int_t^T A(s)\,ds
\right).
\end{equation}
By transversality $\lambda_{I_1}(T)=0$ and hence $\varphi_1(T)=\lambda_{U_1}^*\leq 0$. Also, $\beta S(T)\Psi(T)=\beta S(T)$. Therefore, evaluating \eqref{eq:phi1-first-order} at \(T\) gives $\dot \varphi_1(T)=-\beta S(T)<0$. It follows from \eqref{eq:w1-representation} that $\dot\varphi_1(t)<0$ for all $t\in[0,T]$ (a.e.). Thus \(\varphi_1\) is strictly decreasing and can vanish at most once. Since the maximum condition selects \(u_{1\max}\) when
\(\varphi_1>0\) and \(0\) when \(\varphi_1<0\), there exists
\(\tau_1\in[0,T]\) such that, up to equality almost everywhere,
\[
u_1^*(t)
=
\begin{cases}
u_{1\max}, & 0\leq t<\tau_1,\\
0, & \tau_1<t\leq T.
\end{cases}
\]
\end{proof}
\begin{proposition}[The control $u_2$]
\label{prop:u2-one-switch}
The switching function \(\varphi_2\) is strictly decreasing. Consequently,
\(u_2^*\) has at most one switch, necessarily from \(u_{2\max}\) to \(0\).
\end{proposition}
\begin{proof}:
The maximum condition gives
\[
u_2^*(t)
=
\begin{cases}
u_{2\max}, & \varphi_2(t)>0,\\
0, & \varphi_2(t)<0,
\end{cases}
\qquad\text{for a.e. }t.
\]
Differentiating $\varphi_2$ gives
\begin{align}
\dot\varphi_2
&=
-p\dot\lambda_{I_2}\notag\\
&=
-p[\beta\, S\, \Psi+\gamma \lambda_{I_2}].\notag
\end{align}
Again the right-hand side is continuous and therefore \(\varphi_2\in C^1([0,T])\). We can differentiate almost everywhere to obtain
\begin{align}
\ddot \varphi_2
&=
-p[-\beta\, S\, \dot \lambda_{I_1}+\gamma \dot \lambda_{I_2}]\notag\\
&=
-p \beta S \dot \varphi_1+\gamma \dot \varphi_2.
\end{align}
By Proposition~\ref{prop:u1-one-switch}, $\dot\varphi_1<0$, and therefore, re-arranging and multiplying by the integrating factor $e^{-\gamma t}$ we have 
\[
\frac{d}{dt}[e^{-\gamma t}\dot \varphi_2] =-p\beta\, S\, \dot \varphi_1e^{-\gamma t}>0
\]
and so, $e^{-\gamma t}\dot \varphi_2$ is strictly increasing in time. As a result, $e^{-\gamma t}\dot \varphi_2(t)<e^{-\gamma T}\dot \varphi_2(T)$ or 
\begin{equation}
\dot \varphi_2(t)<e^{-\gamma (T-t)}\dot \varphi_2(T).
\end{equation}
Lastly, we also have that $\dot \varphi_2(T)=-p \beta S(T)<0$ and therefore $\dot \varphi_2(t)<0$ for all $t\in[0,T]$. Thus \(\varphi_2\) is strictly decreasing and can vanish at most once. The maximum condition therefore implies that there exists
\(\tau_2\in[0,T]\) such that, up to equality almost everywhere,
\[
u_2^*(t)
=
\begin{cases}
u_{2\max}, & 0\leq t<\tau_2,\\
0, & \tau_2<t\leq T.
\end{cases}
\]
\end{proof}
\subsection{Numerical results}
Figure \ref{fig:numsolns} shows an example numerical solution to the optimal control problem, where panels G and H illustrate the result of Theorem 1, that initially the optimal controls take their maximum rates and a single switch occurs when all the resources are used. 

\begin{figure}[h]
\centering
\includegraphics[width=.8\textwidth]{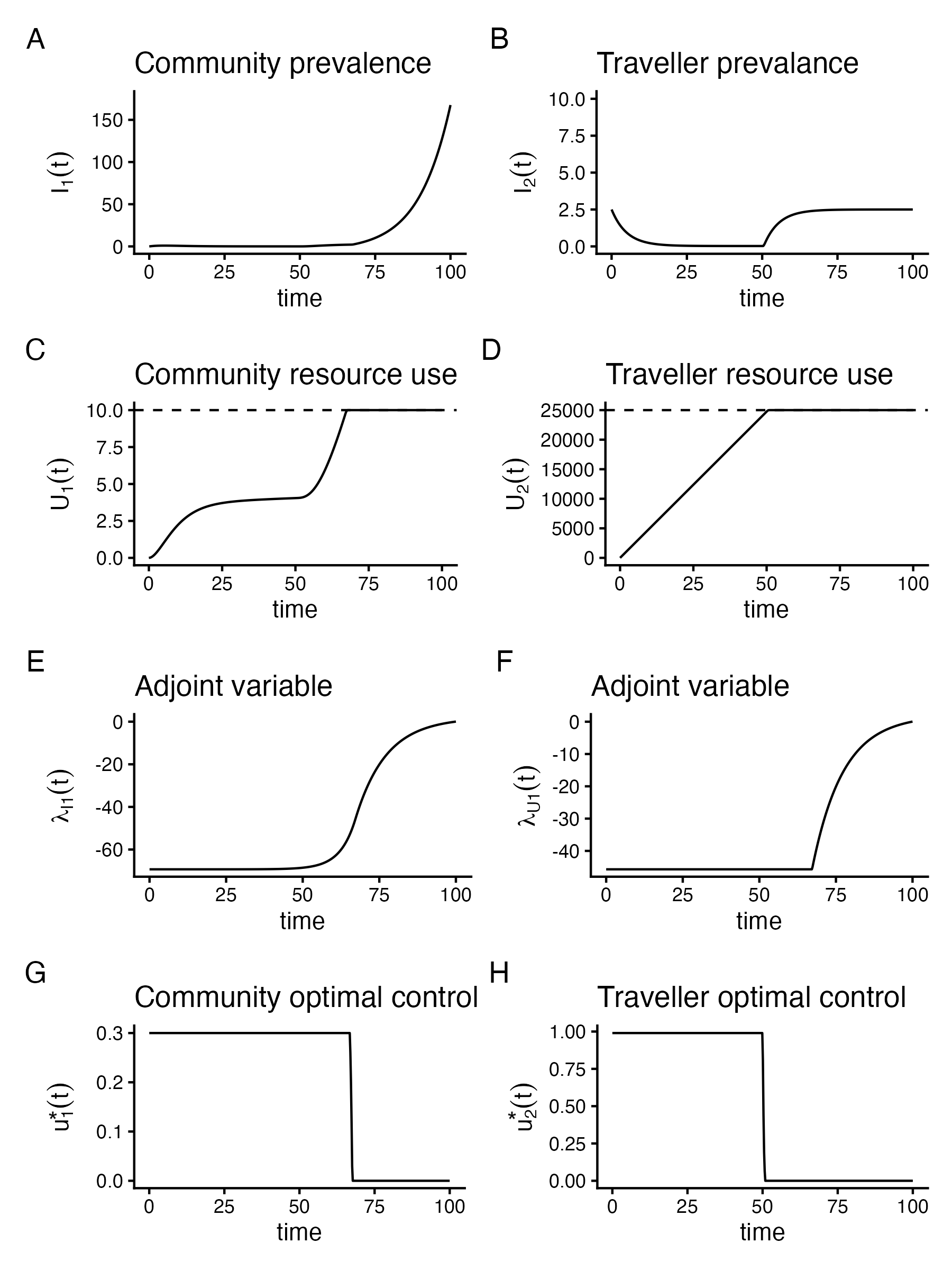}
\caption{The optimal controls (G,H) are determined by the systems of equations \eqref{mm1}-\eqref{eq:U2-} (with their initial conditions; some of which and shown in A-D) and equations \eqref{eq:ad1}-\eqref{eq:adend} (with their boundary conditions as defined by the transversality conditions; some of which are shown in E,F). Parameter values are the default parameter values (Table \ref{tab:params}) except that $u_{1 \max} = 0.3$ per day, $U_{1\text{max}}=30$ individuals, $U_{2\text{max}}$ $=0.5 \times T \times \theta$, and $T=100$ days.}\label{fig:numsolns}
\end{figure}

To test the effect of delaying the introduction of controls, we solved the model defined by Eqs. \eqref{mm1}-\eqref{eq:U2-} where the controls $u_1(t)$ and $u_2(t)$ are both initially zero, are switched to their maximum value at some start time, and are switched off when resources are exhausted. Figure \ref{fig:PHstrategies} shows the results for four different parameter combinations, which could represent regional differences, leading to different epidemiological outcomes: (A) elimination/suppression; (B) mitigation; (C) resurgence after elimination/suppression; (D) resurgence after mitigation. The top panels of each of A-D in Figure \ref{fig:PHstrategies} show the value of the objective function (i.e. total number of infections) across a range of start times for community controls (horizontal axis) and traveller controls (vertical axis) from day 0 to day 300 in 7 day increments.

When the parameters representing local public health resources or region-specific characteristics that affect the epidemiology are such that an elimination/suppression, mitigation, or resurgence after elimination occurs, the percentage of the population that is infected if the control is implemented immediately is much less than if control measures are delayed. This can be observed as the color bar in Figure~\ref{fig:PHstrategies}A-C ranges from 0 to 80$\%$ or more. When resurgence occurs after mitigation, there is less difference between the immediate implementation strategy and strategies that delay the implementation of the control measures. This can be observed as the colorbar in Figure~\ref{fig:PHstrategies}D covers a smaller range: from 65 to 85$\%$. Delaying the implementation of both controls by 28 days increases the percentage of the population infected before the vaccine or therapy is developed from 0.7$\%$ to 3.0$\%$ (A), 29.4$\%$ to 32.0$\%$ (B), 3.0$\%$ to 80.0$\%$ (C), and 68.1$\%$ to 70.0$\%$ (D).

The weak dependence of $J$ on $u_{2,\mathrm{start}}$ (the vertical axis in the upper panels of Figure~\ref{fig:PHstrategies}A-D) indicates that the time of starting the traveller control measure has little relative impact on the objective function, which is the percentage of people infected during the $T$ day period until the vaccine or therapy is developed. However, this illustrates that delaying community measures has a more substantial impact on the objective function, but it does not mean that travel measures are unimportant as we will discuss subsequently. 

For our model, there is no strategy that does better than implementing both controls at their maximum values on day 0 until they are all used up (Theorem 1). We term this the `immediate implementation' control strategy. For the simulations in Figure \ref{fig:strategy} we consider only the immediate implementation structure of the controls and perform numerical simulations to understand how available public health resources and regional factors impact the categorization  of feasible public health strategies (as defined in Table~\ref{tab:strategy-defns}).

\begin{figure}[h]
\centering
\includegraphics[width=1\textwidth]{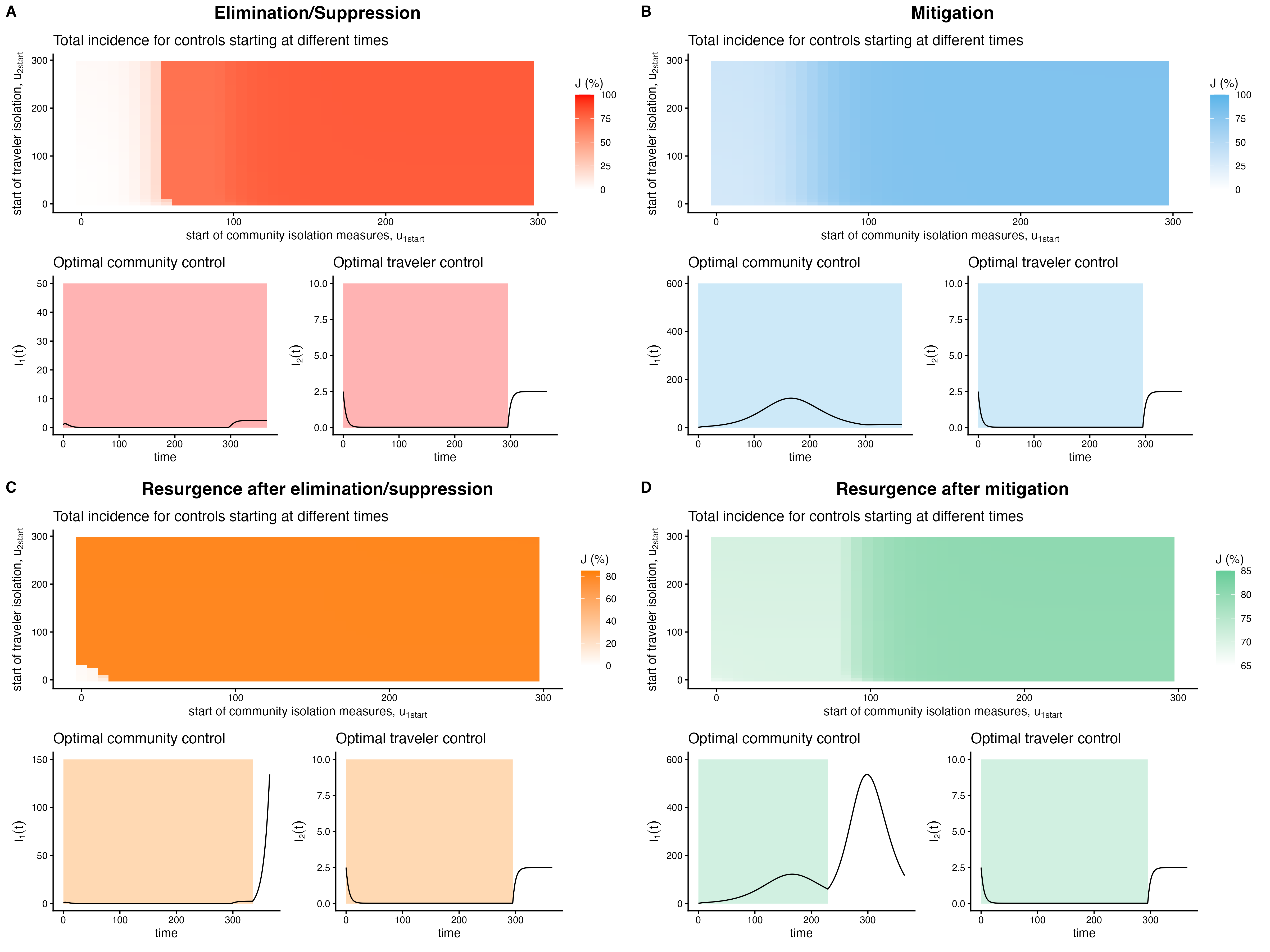}
\caption{Examples of epidemiological dynamics consistent with four possible public health strategies: (A) elimination/suppression, (B) mitigation, (C) resurgence after elimination/suppression, and (D) resurgence after mitigation. The top panels of A-D show the percentage of the population that has been infected during the period $[0,T]$ for community and traveller isolation measures starting at different times. For all of A-D, starting the controls on day 0 (immediate implementation) results in the smallest percentage of the population infected. The lower panels of A-D show community infection prevalence, $I_1(t)$, and infection prevalence in travellers that do not self-isolate, $I_2(t)$, given immediate implementation. The shaded regions indicate $u_1(t) = u_{1\text{max}}$ or $u_2(t) = u_{2\text{max}}$ and the unshaded regions indicate $u_1(t) = 0$ and $u_2(t) = 0$. All of the parameters have the values listed in Table \ref{tab:params} except that: (A) $u_{1\text{max}} = 0.3$ per day; (B) $u_{1 \max} = 0.07$ per day and $U_{1\text{max}}=5000$ individuals; (C) $u_{1\text{max}} = 0.3$ per day and $U_{1\text{max}}=30$ individuals; and (D) $u_{1 \max} = 0.07$ per day. For (A) and (D), $U_{1\text{max}}$ has its default value of 1000 individuals.}\label{fig:PHstrategies}
\end{figure}

We needed to increase the total amount of resources available for isolating community members, $U_{1\text{max}}$, from the default value of 1000 individuals to 5000 individuals, to demonstrate a mitigation strategy without resurgence (Figure \ref{fig:PHstrategies}B). When $u_{1max}$ is too small to achieve elimination or suppression, resurgence often occurs because many community isolation resources are used during a mitigation strategy and the total number of resources is more quickly exhausted. This result is further supported by our next simulations, which find that mitigation (without resurgence) occurs only for a narrow range of public health resources and local epidemiological condition values.

\begin{figure}[h]
\centering
\includegraphics[width=1\textwidth]{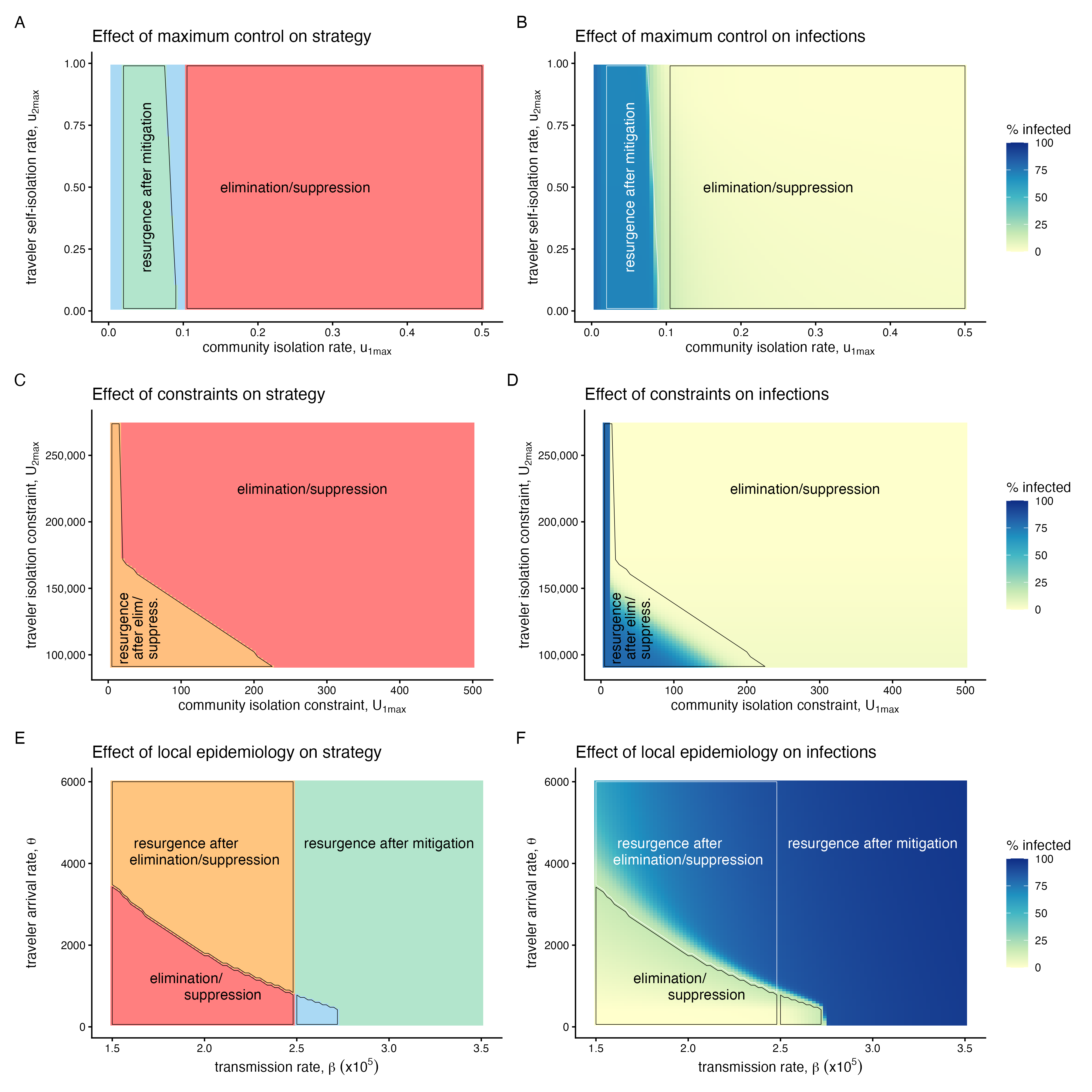}
\caption{Local public health resources (A,C) and local epidemiology (E) determine the best public health strategy and B,D and F show the percentage of the population infected in each of these situations. In A and E mitigation occurs as shown by the light blue fill colour (unlabelled due to small size). All parameters have their default values (Table \ref{tab:params}) except when parameters appear on the axes.}\label{fig:strategy}
\end{figure}

Figure \ref{fig:strategy} demonstrates that the categorization of the epidemiology in terms of public health strategies depends on available public health resources (A,C) and regional variables (E). In Figure \ref{fig:strategy}A, there is a threshold value of the maximum rate of community member isolation $u_{1\text{max}}^{\text{thres}} = \beta S(0) - \mu = 0.1$ per day (corresponding to $R_0^C=1$) such that for $u_{1\text{max}}^{\text{thres}} > u_{1\text{max}}$ a mitigation strategy (with or without resurgence) occurs and for $u_{1\text{max}}^{\text{thres}} < u_{1\text{max}}$ an elimination/suppression strategy occurs. Only a very small region of parameter space finds mitigation without resurgence (light blue in Figure \ref{fig:strategy}A,E) occurs because often community resources are exhausted before time $T$ when mitigation occurs, leading to mitigation with resurgence (green in Figure \ref{fig:strategy}A). Similarly, we do not find that resurgence occurs after elimination/suppression (i.e., there is no orange region in Figure \ref{fig:strategy}A). This is because when elimination/suppression occurs, there are not a lot of community isolation resources used because there are few total infections (Figure \ref{fig:strategy}B). 

In addition to considering differences in local public health resources as these affect the maximum rates and proportions, we also consider differences in the total amount of public health resources available during the $T$ day period until a vaccine or treatment is discovered. We find that resurgence after elimination/suppression occurs when the total amount of resources available to support isolation of community members, $U_{1\text{max}}$, or the total amount of resources available to support traveller compliance with self-isolation requirements, $U_{2\text{max}}$, is small (Figure~\ref{fig:strategy}C). When $U_{1\text{max}}$ and $U_{2\text{max}}$ are large enough to be sufficient to last until time $T$, we find that elimination/suppression occurs. For the default parameters, $R_0^C = 0.8$ and so Figure \ref{fig:strategy}C illustrates elimination/suppression with or without subsequent resurgence because $U_{1\text{max}}$ and $U_{2\text{max}}$ do not impact whether elimination/suppression or mitigation occurs. If the default values corresponded to $R_0^C > 1$, the strategies that were shown to depend on different values of $U_{1max}$ and $U_{2max}$ in Figure~\ref{fig:strategy}C would be mitigation with, and without, resurgence. The percentage of the community that is infected prior to the discovery of the vaccination or treatment is much smaller when the total resources are sufficient to last until $T$ (Figure~\ref{fig:strategy}D).

In contrast to Figure~\ref{fig:strategy}A,B which finds that the maximum proportion of travellers that comply with self-isolation measures does not substantially impact the optimal strategy or the percentage of the community that are infected, we find in Figure~\ref{fig:strategy}C,D that the total amount of resources available to support traveller compliance with self-isolation measures can impact both whether resurgence occurs and the percentage of the community that is infected. This is because when traveller isolation resources are exhausted, more community members are infected by travellers, which consumes community isolation resources, and can result in resurgence if the total resources to support isolation of community members are exceeded. Therefore, we find that the most substantial impact of traveller isolation measures is to help ensure that the capacity to implement community measures is not exhausted, because exhausting these resources results in resurgence.

The effect of regional factors on the categorization of the optimal public health strategy is shown in Figure \ref{fig:strategy}E. In Figure \ref{fig:strategy}E, there is a threshold value of $\beta^{\text{thres}} = 2.5 \times 10^{-5}$ (corresponding to $R_0^C=1$) such that for $\beta^{\text{thres}}>\beta$ elimination/suppression occurs and for $\beta^{\text{thres}}<\beta$ mitigation occurs. As the traveller arrival rate, $\theta$, increases, then resurgence may also occur as community isolation resources are more quickly exhausted when community members are more often infected by travellers. Mitigation can occur when $\beta$ is slightly larger than $\beta^{\text{thres}}$ and $\theta$ is small. However, once $\beta$ or $\theta$ becomes very large, then resurgence occurs after mitigation because many infections occur and this results in all the resources being used. For all the values considered in Figure \ref{fig:strategy}B,D,F, we find that if resurgence occurs, then a much higher percentage of the community is infected.

\section{Discussion}\label{sec12}

We have investigated alternative epidemic response strategies in a model with imported infections and constraints on the public health resources available to support isolation of community cases and quarantine or self-isolation of arriving travellers. Our model assumes that there is a fixed amount of support that can be given for the implementation of control measures and the cost of community control measures is proportional to time-varying infection prevalence. We assume that community control occurs by implementing targeted TTI measures as opposed to blanket PHSMs, such as school or business closures, that apply to the whole population. The distinction is important because for population-level PHSMs, the resource cost is largely independent of prevalence, which leads to optimal mitigation solutions delaying the introduction of interventions until prevalence is sufficiently high (e.g. \citealt{britton}). In contrast, for our model assumptions we proved that there is nothing to be gained by delaying the introduction of TTI interventions at maximum intensity and in some cases delaying interventions makes the outbreak substantially worse.  

We found that our model resulted in different outcomes depending on epidemiological and available public health resources, where differences in these values might be understood to represent different countries or regions. When the maximum intensity of control measures and the available resources were high, controls were sufficient to achieve elimination or suppression of the disease until time $T$, however, resurgence occurred if resources were exhausted before time $T$. If the maximum intensity of control measures was not sufficient to reduce the effective reproduction number below $1$, mitigation occurred (and was typically followed by resurgence as mitigation uses resources very quickly due to the high incidence).

Some previous work has found that imported infections are likely to contribute little to local epidemics \citep{russell_effect_2021, hollingsworth_will_2006, gumel_modelling_2004, arino_risk_2021, chinazzi_effect_2020, wells_impact_2020}. This may be true once an epidemic is established and community prevalence is sufficiently high, however, our model shows that, when community prevalence can be limited to low levels by case-targeted measures, travel controls can have an important effect. Since there are only a fixed amount of resources available, a higher rate of imported infections can mean that the number of community members infected by travellers can be enough to exceed the local capacity for TTI, which results in a large epidemic that could otherwise be delayed until a vaccine becomes available. Through this mechanism, the impact of travel measures can be substantial. In their report to the House of Assembly, Newfoundland and Labrador Department of Health and Community Services stated that the purpose of travel restrictions was to `help detect, trace, and prevent the spread of imported cases of COVID-19 efficiently and effectively by ensuring traveller compliance with self-isolation requirements' (\citealt{HOA}, p20). Similarly, New Zealand's Royal Commission of Inquiry into Covid-19 found that border restrictions were a key component of the elimination strategy, which successfully reduced the impact of the pandemic \citep{NZRCOI_phase1,NZRCOI_phase2}. 

Our analysis of how the feasibility of public health strategies can depend on regional factors is partly motivated by the World Health Organization's \emph{`Technical considerations for implementing a risk-based approach to international travel in the context of COVID-19: Interim guidance, 2 July 2021'} \citep{who_technical_nodate}. This states that travel volumes, local epidemiology, public health measures and capacity, and contextual factors such as the economic and human rights impact determine whether implementing international travel measures during a pandemic should be considered. This recommendation relates to our study as travel measures are often a component of an elimination strategy.

The value of our mathematical modelling approach is to tangibly describe the situations in which different public health strategies are feasible. While the problem set-up involves the assumptions typically associated with ordinary differential equation model formulations, the results are qualitative, i.e., elimination or mitigation strategies result. The details of any specific application might change quantitative nuances of the feasible public health strategies, but the general principles that we describe are likely applicable in many settings and contribute insights into important public health problems.

Generally, the terminology of elimination, suppression, and mitigation strategies, as potential alternatives, is challenging to work with. Few studies aim to unite the terminology, and those that do, e.g. \cite{baker_elimination_2020,baker2023covid}, have not done so from the perspective of modelling and mathematical definitions. To illustrate the challenges, consider a jurisdiction that has no hospitals with intensive care units. For this jurisdiction, a mitigation strategy (so as to not exceed ICU capacity) likely entails implementing an elimination or suppression strategy (reducing infection prevalence to be low or zero). From the perspective of the terminology, it is not helpful if two regions are both said to be implementing a mitigation strategy, but one with a high healthcare capacity is managing to achieve moderate infection prevalence, while the other with low healthcare capacity is managing to achieve elimination, which from the perspective of the public health measures implemented can look quite different. Public health units might often decide on public health measures without knowing if elimination, suppression, or mitigation is the likely outcome of the implemented measures. Our results show that the differences in the percentage of the population infected for different available public health resources can be very substantial, and so it would be advisable for public health units to seek scientific or modelling support if it is not known whether different potential control strategies would result in elimination, mitigation, or resurgence.

An implication of our results is to motivate more careful comparisons of the effectiveness of the COVID-19 response in different countries. In \cite{peng2023relative}, Canada and Australia are described as peer-countries due to similarity in `income, culture and governance' but are noted to differ in that Australia implemented an elimination strategy, while Canada implemented a mitigation strategy. \cite{peng2023relative} conclude that `a comparison with Australia demonstrated that an elimination focus would have saved Canada tens of thousands of lives as well as substantial economic costs'. However, we know that elimination may not be possible (Figure~\ref{fig:strategy}; see also \citealt{martignoni_is_2024}) in regions with a high importation rate and Canada differs substantially from Australia in this respect. Canada has a substantial amount of trade across the United States border. For example, in 2010, 28.9 million cars and 5.4 million trucks crossed the Canada-US border \citep{transport_canada_2024}, with the flow of goods and services estimated at $\$$1 million United States dollars per minute. As Australia is an island, it has a substantially lower importation rate than Canada, and more readily controllable borders. Therefore, it is possible that an elimination strategy in Canada would not have been feasible, or that it would not have achieved the same reduction in deaths due to COVID-19 as Australia had the strategy been attempted.

Because our model only considers targeted TTI measures and not broader PHSMs, our simulations result in mitigation whenever TTI measures alone are insufficient to reduce the effective reproduction number below $1$. In some situations, it may be possible to achieve elimination/suppression by combining TTI with additional PHSMs. Thus, our results cannot be used to conclude that elimination is not feasible in general, only that it is not feasible with TTI measures alone. Determining whether the use of population-scale PHSMs to support elimination is preferable to a mitigation strategy would require the costs and benefits of alternative strategies to be compared \citep{plank2025joint}. Future work will consider this question using a different modelling framework that allows the costs of infections to be traded off against the costs of community and traveller controls.  

Our work is a starting point in understanding how regional differences can impact the feasibility of public health strategies, and illustrates that between-country comparisons to determine if a different approach to managing COVID-19 might have been better for a particular country should consider travel in addition to economic and social factors that may determine public health resource limits. We also find that regional factors can interact to determine the feasible public health strategies.  Lower traveller arrival rates and greater public health capacity to support the isolation of travellers reduces the amount of community isolation resources needed. In this way, travel measures can protect contact tracing capacity and prevent epidemic resurgence.

\backmatter


\subsection*{Declarations}

\begin{itemize}
\item Funding. G. Adu-Boahen and A. Hurford were supported by the OMNI-REUNIS, which was funded by the Natural Sciences and Engineering Research Council of Canada (NSERC) Emerging Infectious Disease Modelling Initiative. A. Hurford is supported by an NSERC Discovery Grant (RGPIN-2023- 05905). M.J. Plank was supported by the Marsden Fund (grant number 24-UOC-020) and Te Niwha Infectious Diseases Research Platform (grant number TN/P/24/UoC/MP), co-hosted by PHF Science and the University of Otago and provisioned by the Ministry of Business, Innovation and Employment, New Zealand.
\item Conflict of interest/Competing interests (check journal-specific guidelines for which heading to use)
\item Code availability. \href{https://github.com/ahurford/elimination-mitigation}{https://github.com/ahurford/elimination-mitigation}.
\item Author contribution. G Adu-Boahen: model design, analysis, coding, and writing; T. Day: model design, analysis; M. Plank: model design, analysis, coding and writing; A. Hurford: model design, analysis, coding, writing, funding acquisition, and supervision.   
\end{itemize}

\bibliographystyle{sn-vancouver-ay}
\bibliography{references}

\begin{appendices}
\newpage
\setcounter{page}{1}
\renewcommand{\theequation}{A.\arabic{equation}}
\setcounter{equation}{0}
\renewcommand{\thesection}{A.\arabic{section}}
\setcounter{section}{0}




\end{appendices}

\end{document}